\documentclass[reqno]{amsart}

\usepackage{amsmath,amssymb,color,amsthm}
\usepackage{enumerate,enumitem}
\usepackage{booktabs,tabularx,multirow}
\usepackage{graphicx,subfigure}
\usepackage{verbatim,mathrsfs}
\usepackage{cite}
\usepackage[margin=1.5in]{geometry}

\DeclareGraphicsExtensions{.pdf}

\theoremstyle{plain}
\newtheorem{theorem}{Theorem}
\newtheorem{lemma}[theorem]{Lemma}

\newtheorem{conj*}[theorem]{Conjecture}

\newtheorem*{claim*}{Claim}

\theoremstyle{definition}

\theoremstyle{remark}

\newtheorem*{remark*}{Remark}

\newcommand{\NN}{\mathbb{N}}

\newcommand{\HH}{\mathcal{H}}
\newcommand{\EE}{\mathcal{E}}
\newcommand{\E}{\mathcal{E}}

\newcommand{\chicf}{\chi_{\text{cf}}}

\newcommand{\Exp}{{\text{E}}}
\newcommand{\Prob}{\text{Prob}}

\begin{document}

\title{Brooks Type Results for Conflict-Free Colorings  and $\{a, b\}$-factors in graphs.}

\author[M. Axenovich]{Maria Axenovich}
\address{Karlsruher Institut f\"ur Technologie, Karlsruhe, Germany}
\email{maria.aksenovich@kit.edu}

\author[J. Rollin]{Jonathan Rollin}
\address{Karlsruher Institut f\"ur Technologie, Karlsruhe, Germany}
\email{jonathan.rollin@kit.edu}

\date{\today}

\maketitle

\begin{abstract}
A vertex-coloring of a hypergraph is {\it conflict-free}, if each edge contains a vertex whose color is not repeated on any other vertex of that edge.
Let~$f(r, \Delta)$ be the smallest integer~$k$ such that each~$r$-uniform hypergraph of maximum vertex degree~$\Delta$ has a conflict-free 
coloring with at most~$k$ colors. 
As shown by Tardos and Pach, similarly to a classical  Brooks' type theorem for hypergraphs,~$f(r, \Delta)\leq \Delta+1$. 
Compared to  Brooks' theorem, according to which there is  only a couple of graphs/hypergraphs that attain the~$\Delta+1$ bound, 
we show that there are several infinite classes of uniform hypergraphs for which the upper bound is attained. 
We provide bounds on $f(r, \Delta)$ in terms of~$\Delta$ for large~$\Delta$ and establish the connection between 
conflict-free colorings and so-called ~$\{t, r-t\}$-factors in~$r$-regular graphs. Here, a~$\{t, r-t\}$-factor is a factor in which each degree is either~$t$ or 
$r-t$. Among others, we disprove a  conjecture of Akbari and Kano~\cite{Kano_factorsRegGraphs} stating that there is a~$\{t,r-t\}$-factor in every~$r$-regular graph for odd~$r$ and any odd~$t<\frac{r}{3}$.
 \end{abstract}

\section{Introduction}

Given a vertex coloring of a hypergraph,  we call a vertex contained in an edge~$E$ \emph{uniquely colored in~$E$}, if its color is assigned to no other vertex in~$E$.
If every edge of a hypergraph contains a uniquely colored vertex, then the coloring is called \emph{conflict-free}.
The \emph{conflict-free chromatic number}~$\chicf= \chicf(\HH) $ of a hypergraph~$\HH$ is the minimum number of colors used in a conflict-free coloring of~$\HH$.

Conflict-free colorings are closely related to proper colorings in which each hyperedge is not monochromatic, i.e., contains at least two vertices of distinct 
colors.  The chromatic number~$\chi(\HH)$ of a hypergraph~$\HH$ is the smallest number of colors in a proper coloring of~$\HH$.
The following Brooks' type theorem for hypergraphs proved by Kostochka, Stiebitz and Wirth~\cite{HypergraphBrook}
generalizes Brooks' theorem for graphs. The maximum degree~$\Delta(\HH)$ of a hypergraph is the largest number of hyperedges containing a common vertex.
A hypergraph is connected if for any two vertices $u, v$ there is a sequence of vertices starting with $u$ and ending with $v$ such that any two consecutive vertices in this sequence
belong to a hyperedge.

\begin{theorem}[\cite{HypergraphBrook}]\label{theorem::BrooksHypergraphs}
 Let~$\HH=(V, \EE) $ be a connected hypergraph with~$|E|\geq 2$ for each edge~$E$.
 If~$|\EE|> 1$ and~$\HH$ is neither an ordinary odd cycle nor an ordinary complete graph, then~$\chi(\HH)\leq \Delta(\HH)$.
 Otherwise~$\chicf(\HH)=\chi(\HH)=\Delta(\HH)+1$.
\end{theorem}

Pach and Tardos provided  the same upper bound on~$\chicf$:
\begin{theorem}[\cite{GeneralConflictFree}]\label{theorem::GeneralUpperChicf}
For each hypergraph ~$\HH$, ~$\chicf(\HH)\leq \Delta(\HH) +1.$
\end{theorem}

In this paper,  we find infinite classes of  hypergraphs for which the upper bound~$\Delta+1$ on the conflict-free chromatic number 
is attained and provide general  bounds in terms of the maximum degree for uniform hypergraphs.
In addition, we provide a connection between conflict-free colorings of hypergraphs and so-called ~$\{t, r-t\}$-factors in~$r$-regular graphs. Here, a~$\{t, r-t\}$-factor is a spanning subgraph  with every vertex of  degree ~$t$ or  $r-t$.  

Next we list the main results of the paper.   A hypergraph is~$r$-uniform if each hyperedge contains exactly~$r$ vertices.
The main parameter we introduce is 

$$f(r, \Delta) = \max \{\chicf(\HH):  ~~ \HH \mbox{ is }  r\mbox{-uniform},~~  \Delta(\HH)=\Delta\}.$$

In Theorem \ref{general-bounds}  we provide  bounds on $f(r, \Delta)$ when $r$ is fixed and $\Delta$ is large. 
Theorem \ref{main} provides Brook's type results.

 \begin{theorem}\label{general-bounds}
 For any integer  $r\geq 3$, there are constants $c, c'$, and $d$  depending on $r$ such that for any $\Delta > d$
 $$c   \frac{\Delta^{1/\lceil \frac{r}{2} \rceil}}{\ln(\Delta)} \leq   f(r, \Delta) \leq c' \Delta^{1/\lceil \frac{r}{2} \rceil}.$$
\end{theorem}

\begin{theorem}\label{main}
Let~$r,\Delta$ be positive integers.
\begin{enumerate} %[label=\ref*{main}.\arabic*)]
\item If~$r=2$ or~$\left( \Delta=2  ~\mbox{ and }  r\not\in \{3,5\}\right)$, then~$f(r, \Delta) = \Delta+1$.\label{maintheorem::DeltaPlusOneColor}
\item  There are constants $c, c_0$ such that if $r\geq c $ and $\Delta \geq c_0  \ln r $ then $f(r, \Delta) \leq \Delta$.
\item  If $\Delta \geq 3$, then $f(4, \Delta) \leq \Delta$. 
\label{maintheorem::rColors}
\end{enumerate}
\end{theorem}

Part of the above theorem is a corollary of an independent result:

\begin{theorem}\label{theorem::deg2}
 Let~$t$ be an odd positive integer.
 If~$r$ is even or~$r\geq (t+1)(t+2)$, then there is an~$r$-regular graph that has no ~$\{t,r-t\}$-factor.
\end{theorem}

This disproves the following conjecture.
\begin{conj*}[Akbari and Kano~\cite{Kano_factorsRegGraphs}]\label{conj}
Let~$r$ be an odd integer and~$a,b$ denote positive integers such that~$a + b = r$ and~$a < b$. Then each~$r$-regular  graph has an~$\{a,b\}$-factor. 
\end{conj*}

Note that $\chicf(\HH)=\chi(\HH)$ if each edge of $\HH$ has size $2$ or $3$.
Hence Brooks' theorem and the theorem by Kostochka et al. \cite{HypergraphBrook} give a complete characterization of $2$- and $3$-uniform hypergraphs $\HH$
for which $\chicf(\HH)= \Delta(\HH)+1$.
Here, we can also characterize the case of $4$-uniform hypergraphs.

\begin{theorem}\label{r=4}
If $\HH$ is a connected $4$-uniform hypergraph on $m$ edges, then $\chicf(\HH)= \Delta(\HH)+1$ iff
$\Delta(\HH)=1$ or  ($\HH$ is $2$-regular and $m$ is odd).
\end{theorem}

We provide some  history of the problem, background, and basic lemmas in Section \ref{back}. An important construction is given in Section \ref{constr}.
Finally, the theorems are proved in Sections $4$-$6$.

\section{Background and previous results}\label{back}

Conflict-free  colorings  were  introduced by Even, Lotker, Ron and Smorodinsky~\cite{ConflictFreeStart, SmorodinskyPhD} when considering a frequency assignment problem in wireless networks.
A set of base stations in the plane defines the vertices of the hypergraph.
A hyperedge is formed by every subset of base stations which is simultaneously reached by a mobile agent from some point in the plane.
In order to avoid interferences in communication there should be at least one base station reachable whose frequency is unique among all stations in the range.
This corresponds to a conflict-free coloring of the underlying (geometric) hypergraph, where colors correspond to frequencies.
Since frequencies are expensive due to limited bandwidth, as few frequencies (respectively colors) as possible shall be used.
A survey by Smorodinsky~\cite{ConflictFreeSurvey} summarizes several results where the hypergraph under consideration is induced by some geometric setting, for example by discs~\cite{ConflictFreeStart, PachToth_ConflictFree} or rectangles in the plane~\cite{Smoro_CFGeometricHypergraphs} and pseudodiscs~\cite{Smoro_HarPeled_CFPointsRegions}.
\bigskip

Besides these geometrically induced hypergraphs other classes of hypergraphs were considered.
Pach and Tardos~\cite{GeneralConflictFree} considered the so called \emph{conflict-free chromatic parameter} of a graph~$G$
defined as the conflict-free chromatic number of the hypergraph on the same vertex set $V(G)$ and hyperedges consisting of vertex neighborhoods in the graph. Initially this and lots of related questions were studied by Cheilaris~\cite{Cheilaris_PhD}.
Recently Glebov, Szab{\'o} and Tardos~\cite{Glebov_ConflictFreeRandom} showed upper and lower bounds on this parameter for random graphs in terms of the domination number.
% \bigskip
In addition,  conflict-free colorings of hypergraphs whose hyperedges correspond to simple paths of a given  graph were considered ~\cite{UniqueMaxVSConflictFree, CheilarisToth_UMPath}.
A variant of conflict-free colorings using  integral colors and  each hyperedge containing  a vertex whose color  is larger than the color of every other vertex in this edge  is called  a Unique-Maximum coloring or  a vertex ranking, see \cite{EdgeRankingIntro, Bodlander_Rankings, EdgeRankingLinear}.

Research on general hypergraphs was started by Pach and Tardos~\cite{GeneralConflictFree}.
They gave several upper bounds on~$\chicf$ in terms of different parameters.
The degree of an edge $E$ in a hypergraph $\HH$ is the number of edges intersecting $E$, and the \emph{maximum edge-degree} $D(\HH)$ is the maximum of the degrees of edges in $\HH$.
For example Pach and Tardos prove~$\chicf \in O(t\cdot D^{\frac{1}{t}}\cdot \ln(D))$ if every edge has size at least~$2t-1$.
More bounds are due to Kostochka, Kumbhat and \L{}uczak~\cite{CFFewEdges}.
\begin{theorem} [\cite{CFFewEdges}]\label{KK}
If $D$ is the maximum edge-degree of an $r$-uniform hypergraph $\HH$, $D$ is sufficiently large  and $D\leq 2^{r/2}$,  then $\chicf(\HH) \leq 120 \ln D$.
\end{theorem}

Pach and Tardos described  a greedy algorithm to find a conflict free coloring with  at most~$\Delta+1$ colors.
We present an alternative proof based on so called \emph{strongly independent} sets of vertices, where a set~$S$ is strongly independent if each edge contains at most one vertex from~$S$.

\begin{proof}[Proof of Theorem \ref{theorem::GeneralUpperChicf}]
If~$\Delta=1$, then the hyperedges are disjoint and two colors are sufficient.
For ~$\Delta>1$, let~$S$ be a maximal strongly independent vertex set in~$\HH$.
The new  hypergraph~$\HH'$   obtained by removing~$S$ and all edges incident to~$S$ has maximum degree at most~$\Delta-1$.
 A conflict-free coloring of~$\HH'$  obtained by induction and using~$\Delta$ colors, can be 
 extended to   a conflict-free coloring of~$\HH$   by assigning a new color to each vertex in~$S$.
 Indeed, each hyperedge from~$\HH$ that is not a hyperedge of~$\HH'$ contains exactly one vertex from~$S$ and 
 thus a uniquely colored vertex.
\end{proof}

The usual and the conflict-free chromatic number may be arbitrarily far apart.
We construct a hypergraph  of chromatic  number~$2$  and conflict-free chromatic number equal to $\Delta+1$  inductively. 
 If~$\Delta=1$, take a single edge with at least two vertices. If~$\HH$ satisfies~$\chi(\HH)=2$ and~$\chicf(\HH)=k$, then construct 
 a hypergraph~$\HH'$   that is a  vertex disjoint  union of two copies of~$\HH$ and an additional edge~$E$ containing all vertices of both copies of~$\HH$.
 Coloring each copy of~$\HH$ properly with~$2$ colors gives a proper coloring of~$\HH'$  with two colors. 
However,  any  conflict-free coloring of~$\HH'$ requires ~$k+1$ colors, since one color is unique in~$E$ and hence is used in only one of the copies of~$\HH$.
 
 There is another easy construction of a hypergraph $\HH$ with maximum degree $\Delta$ and $\chicf(\HH) > \Delta$. 
 Take two vertex disjoint copies of an ordinary complete graph $K_{\Delta}$ and add an edge $E$ containing all the vertices. If $\Delta$ colors are used in a conflict-free coloring of $\HH$ then each copy of $K_{\Delta}$ uses colors $1, 2, \ldots, \Delta$. Thus $E$ uses each of the colors exactly twice, a contradiction.

 Note that the hypergraphs in these  constructions are not uniform.\\

 \noindent
Next, we prove Lemmas that will be used later.\\
 
\begin{lemma}~\label{lem::inductDeltaColorable}
 Let~$r,\Delta\in\NN$. If~$f(r,\Delta)\leq \Delta$, then ~$f(r,\Delta+1)\leq \Delta+1$.
\end{lemma}
\begin{proof}
 Let~$\HH$ be an $r$-uniform hypergraph  of maximum degree~$\Delta+1$ and let~$S$ be  a maximal strongly independent set  in~$\HH$.
 Consider the hypergraph~$\HH'$ which is obtained by removing~$S$ and all edges incident to~$S$ from~$\HH$.
 Then~$\Delta(\HH')\leq \Delta$ since every vertex in~$\HH$ is adjacent to a vertex in~$S$.
Thus $\chicf(\HH') \leq f(r, \Delta) \leq \Delta$ by assumption on $f$.
Extending a conflict-free coloring of $\HH'$ with $\Delta$ colors to a coloring of $\HH$ in which all vertices of $S$ get the same  new color
gives a conflict-free coloring of $\HH$ with at most $\Delta +1$ colors.
\end{proof}

A hypergraph is $a$-regular  if each vertex has degree $a$.
For a hypergraph~$\HH= (V, \EE)$, the dual hypergraph~$\HH' = (\EE, \EE')$ is defined  such that  
$\EE' = \{ E_v: ~ v\in V\}$, ~$ E_v = \{ E\in \EE:   ~ v\in E\}$.
Note that a dual of a~$2$-regular~$r$-uniform hypergraph is an ordinary~$r$-regular graph. Moreover, any $r$-regular graph corresponds to a 
$2$-regular $r$-uniform hypergraph.

\begin{lemma}\label{lem::duality}
Let~$\HH$ be a~$2$-regular $r$-uniform hypergraph and~$G$ be its dual graph.
Then~$\chicf(\HH) = 2$ iff~$G$ has a~$\{1, r-1\}$-factor.
\end{lemma}
\begin{proof}
 Consider a conflict-free~$2$-coloring of~$\HH$ and color each edge in~$G$ with the color of its corresponding vertex in~$\HH$.
 Let~$F$ denote the set of edges of one of the color classes in~$G$.
 Then~$F$ is a~$\{1, r-1\}$-factor.
 The other way round observe that the complement of a~$\{1, r-1\}$-factor in~$G$ is a~$\{1, r-1\}$-factor as well.
 So, $G$ is an edge-disjoint union of $F_1$ and $F_2$, two $\{1, r-1\}$-factors.  Color the vertices of $\HH$ corresponding to edges of $F_1$ red, 
 and the vertices corresponding to the edges of $F_2$ blue.   This defines a conflict-free coloring of~$\HH$.
\end{proof}

The following lemma is a straightforward application of the Local Lemma that was pointed out in most papers on conflict-free colorings. We include it here for completeness.

\begin{lemma}\label{lem::LLLUpperChicf}
 Let~$r,\Delta\in\NN$,~$r,\Delta\geq 2$, and~$\HH=(V, \E)$ be an~$r$-uniform hypergraph of maximum degree~$\Delta$.
Then \[\chicf(\HH)\leq \frac{1}{2} (e r)^{1+ \frac{2}{r}}  \cdot\Delta^{1/\lceil \frac{r}{2} \rceil}.\]
\end{lemma}
\begin{proof} 
 Color  each vertex  uniformly and independently at random  with one out of~$k$ colors.
For a hyperedge $E$, let $A_E$ denote the event that $E$  has at most $\lfloor r/2  \rfloor$ colors. 
We call $A_E$ a bad event.
Note that when there are no bad events,  each edge has more than $\lfloor r/2  \rfloor$ colors, and thus each edge has a vertex of a unique color.
Using  Lov\'asz  Local Lemma, we have $Prob(  \bigwedge _ {E \in \E}  \overline{A_E}) >0$ if 
$Prob (A_E) \leq p$ for each $E$ and $ ped <1$, where each $A_E$ is mutually independent of all but at most $d$ other bad events.
We have that  
$$Prob (A_E) \leq     \binom{k}{\lfloor r/2  \rfloor}   \left(  \left\lfloor \frac{r}{2}  \right\rfloor   \frac{1}{k}   \right)^{r}       \leq     k^{-\lceil \frac{r}{2} \rceil} \left(e\tfrac{r}{2}\right)^{\lceil\frac{r}{2}\rceil}.$$
So, let $p= k^{-\lceil \frac{r}{2} \rceil} \left(e\tfrac{r}{2}\right)^{\lceil\frac{r}{2}\rceil}$.
In addition, there are at most~$r(\Delta-1)$  edges intersecting any given edge $E$, so $d\leq r(\Delta -1)$.
Thus  $ped<1$ if $  e r \Delta k^{-\lceil \frac{r}{2} \rceil}\left(e \tfrac{r}{2}\right)^{\lceil\frac{r}{2}\rceil}<1.$
Solving for~$k$, we have that $ped<1$  if  $k\geq (e r)^{\frac{2}{r}} \tfrac{er}{2} \cdot\Delta^{1/\lceil \frac{r}{2}\rceil}$.
Therefore, there is a conflict-free coloring of $\HH$ with $k$ colors for $k \geq \frac{1}{2} (e r)^{1+ \frac{2}{r}}  \cdot\Delta^{1/\lceil \frac{r}{2} \rceil} $ and thus $\chicf(\HH)\leq 
\frac{1}{2} (e r)^{1+ \frac{2}{r}}  \cdot\Delta^{1/\lceil \frac{r}{2} \rceil}$.
\end{proof}

For additional  notations on hypergraphs we refer to~\cite{Bollobas_Hypergraphs}.
In the remaining text we call an edge \emph{good}, if it contains a uniquely colored vertex and a coloring \emph{good}, if every edge is good.

%%%%%%%%%%%%%%%%%%%%%%%%%%%%%%%%%%%%%%%%%%
%%%%%%%%%%%%   Construction  %%%%%%%%%%%%%%%%%%%
%%%%%%%%%%%%%%%%%%%%%%%%%%%%%%%%%%%%%%%%%%

\section{Construction  of  a graph $G(t, r)$}  \label{constr}

\noindent
Let $r$ and $t$ be positive odd integers, $r\geq (t+1)(t+2)$.
To construct~$G(t, r)$, we shall consider several graphs. 
Start with a copy~$K$ of~$K_{r-1,r}$ with partite sets~$U,V$, where~$|U|=r-1$,~$|V|=r$.
A graph~$H$  is obtained by adding a matching~$M$ on~$r-t-2$ vertices from~$V$ and by adding a new vertex~$u$ connected to the~$t+2$ remaining vertices of degree~$r-1$ in~$V$.  In~$H$ every vertex in~$U\cup V$ is of degree~$r$ and the vertex~$u$ is of degree~$t+2$.

 \noindent
 Take~$t+1$ vertex-disjoint copies~$H_1, \ldots, H_{t+1}$ of~$H$ and identify all copies of the vertex~$u$. 
 Call the resulting  vertex~$w$, and the resulting graph~$G'$.  Moreover, call the copy of~$M$ in~$H_i$ as~$M_i$,  and
 call the copies of~$U$ and~$V$ as~$U_i$ and~$V_i$, respectively, ~$i=1, \ldots, t+1$. 
 The degree of~$w$ in~$G'$ is~$(t+1)(t+2)\leq  r$ and the degree of each other 
 vertex is~$r$.
 Let~$\delta:=r-(t+1)(t+2)$.\\
 
 \noindent
 The graph~$G(t,r)$ is defined by taking~$\delta+1$ copies of~$G'$ and adding an edge between each two copies of~$w$.
 Since every copy of~$w$ receives~$\delta$ additional edges, the graph~$G(t,r)$ is~$r$-regular.
 See Figure~\ref{fig::oddSimpleNoFactor} for an illustration of the case~$r=7$,~$t=1$ and~$\delta=7-2\cdot 3=1$.
This concludes the Construction.

\begin{figure}[tb]
 \label{fig::oddSimpleNoFactor}
 \centering
 \includegraphics[height=4cm]{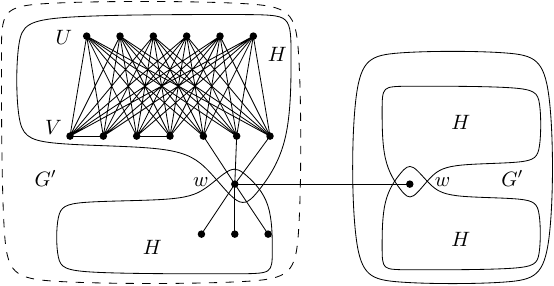}
 \caption{A~$7$-regular graph without~$\{1,6\}$-factor.}
\end{figure}

~\\

\section{Proof of Theorem \ref{general-bounds}}

The  upper bound  follows immediately from   Lemma \ref{lem::LLLUpperChicf}.  For the  lower bound,  we shall also employ a random argument generalizing the theorem of Kostochka et al. \cite{CFFewEdges} to uniform hypergraphs with even or odd number of vertices in each edge, and with bounded maximum degree. We omit ceilings and floors where it is clear from context. We always keep those in the leading terms.
 Let $k=\Delta^{ 1/ \lceil  \frac{r}{2} \rceil  }   \frac{1}{1600\ln(\Delta)}$, 
 $n=6k$ and $m=\frac{n\Delta}{2r}$.  Assume that $\Delta$ is sufficiently large.  We shall consider an $r$-uniform hypergraph on a vertex set $V$, $|V|=n$ with at most $m$ edges by picking $m$ hyperedges uniformly at random (with possible repetitions). \\

Formally, introduce a uniform distribution on elements of $\binom{V}{r}$,   i.e.   choose a set   $F\in\binom{V}{r}$   with probability $p=\binom{n}{r}^{-1}$.
For a sequence of $m$ hyperedges $F_1, \ldots, F_m$, the probability of choosing all of them is  $ p^m$.
So, this gives a distribution on a set of such sequences, each corresponding to a hypergraph on at most $m$ edges.\\

We shall show that with a positive probability the selected sequence of hyperedges gives a hypergraph with maximum degree at most $\Delta$ and conflict-free chromatic number greater than $k$.
Let $\HH$ be a  hypergraph with hyperedge set $\{F_1, \ldots, F_m\}$. \\

 We consider $\Delta(\HH)$ first.    In the following calculations we shall be using the fact that $\Delta \geq k^{\lceil \frac{r}{2} \rceil } 1600^{\lceil \frac{r}{2} \rceil}\ln k \geq k^{\lceil \frac{r}{2} \rceil } 40^{r}\ln k$.
 Further, we assume that $r$ is fixed and $\Delta$ is sufficiently large, i.e., $k$ is sufficiently large.
 For a vertex $v$ let $X_v$ denote the number of  $F_i$s  containing $v$.
 Let $Y_i=1$ if  $F_i$ contains $v$, let $Y_i=0$ otherwise. Then $\Prob(Y_i=1) = {\binom{n-1}{r-1}}/{\binom{n}{r}}= \frac{r}{n}$ and $X_v = Y_1+\ldots + Y_m$. Thus $\Exp(X_v)= m \frac{r}{n} = \Delta/2$.
 
 Using Chernoff  bound  $ \Prob( X_v  > \lambda+  \Exp(X_v)) < \exp(-2\lambda^2/m)$   with  $\lambda = \Delta/2$, we have

 \begin{align*}
\Prob(X_v >\Delta) &= \, \Prob(X_v > \Delta/2 + \Delta/2) \\
& \leq \,  \exp ( -2 (\Delta/2)^2/m) \\
& =  \, \exp\left(-  \frac{r\Delta  }{n}  \right) \\
&\leq   \exp\left( - \frac{r k^{\lceil \frac{r}{2} \rceil } 40^{r}\ln k }{6k}  \right) \\
& \leq  \,  \exp( -      k^{{\lceil \frac{r}{2}\rceil  }-1}).
 \end{align*}   
 
Then
 
  \begin{align} \label{max-degree}
\Prob(\Delta(\HH) > \Delta) & \leq \,  \Prob( X_v>\Delta \mbox{ for some } v) \nonumber  \\ 
 & \leq \, n\, \Prob(X_v> \Delta) \nonumber \\
 & \leq \, n \exp( -k^{{\lceil \frac{r}{2}\rceil  }-1}) \nonumber \\
 & = \, \exp( \ln 6k - k^{{\lceil \frac{r}{2}\rceil  }-1}) \nonumber \\
 &< 1/2.
\end{align}

 Next we consider $\chicf(\HH)$.  Fix a $k$-coloring $\chi$ of $V$. 
 In this part, we use inequalities  $\left( \frac{a}{b} \right)^b  \leq \binom{a}{b} \leq \left(  \frac{ea}{b}\right)^b$.
 As was observed by Kostochka et al.  \cite{CFFewEdges},  there are pairwise disjoint monochromatic sets $A_1, \ldots,  A_k$ each of size $\lfloor \frac{n}{2k} \rfloor = 3$. 
 We say that $F_i$ is bad if it has no uniquely colored vertex.  The probability that $F_i$ is bad  is bounded from below by the probability that $F_i$ has   $2$ or $3$  (depending on parity of $r$) vertices in one of $A_i$, and exactly $2$ or $0$ vertices in the rest of $A_i$s.  In this case the number of $A_i$'s with some elements of $F_i$ is   $\lfloor \frac{r}{2} \rfloor$. Thus, the probability that $F_i$ is bad is at least   the number of ways to choose $\lfloor\frac{r}{2}\rfloor$ of the $A_i$'s (with some elements of $F_i$) multiplied by $p$.  So, 
 \begin{align*}
 \Prob (F_i \mbox{ is bad }) &\geq \, \binom{k}{ \lfloor \frac{r}{2} \rfloor  }  \binom{n}{r}^{-1}\\
 &\geq \,   \left(\frac{k}{\lfloor \frac{r}{2} \rfloor}\right) ^{\lfloor \frac{r}{2}\rfloor}   \left( \frac {r}{e6k}   \right)^r  \\
 &\geq \, C k^{-\lceil  \frac{r}{2} \rceil},
 \end{align*}
 For a positive constant $C=C(r)\geq  \left(\frac{1}{r/2}\right)^{r/2} r^r  \left(\frac{1}{6e}\right)^r \geq   \left(  \frac{ 1}{6e}\right)^r \geq  20^{-r}$.

 Since $\chi$ is a conflict-free coloring of $F_1, \ldots, F_m$ iff each $F_i$ is not bad, we have
 
 $$\Prob(\chi \mbox{ is conflict-free coloring of }  \HH)  \leq  (1 - C k^{-\lceil  \frac{r}{2} \rceil} )^m.$$
 Looking at all $k^n$ possible $k$-colorings of $V$, we have
 
 \begin{align} \label{cf} 
 \Prob( \chi_{cf} (\HH) \leq k) &\leq \,  k^n (1- c k^{-\lceil  \frac{r}{2} \rceil})^m \nonumber\\
 & \leq \,  \exp( n \ln k ) \exp ( -  20^{-r}  k^{-\lceil  \frac{r}{2} \rceil}m)\nonumber \\
 & = \,  \exp\left( 6k\ln k  -  20^{-r} k^{-\lceil  \frac{r}{2}  \rceil }\frac{n\Delta}{2r}\right)\nonumber\\
 & \leq \,  \exp\left( 6k\ln k  -  20^{-r} k^{-\lceil  \frac{r}{2}  \rceil }\frac{6k }{2r}40^r k^{\lceil \frac{r}{2} \rceil } \ln(k)\right)\nonumber\\
  & \leq \,  \exp ( 6k\ln k  -  \tfrac{3}{r} 2^r  k\ln k)  \nonumber\\
  & \leq \,  \exp ( 6k\ln k  -  8  k\ln k)  \nonumber\\
  &< \, 1/2.
 \end{align}

 Putting (\ref{max-degree}) and (\ref{cf}) together, we have that with positive probability there is an $r$-uniform hypergraph with maximum degree at most $\Delta$ and 
 conflict-free chromatic number greater than $k$.

\qed

\section{Proofs of Theorems~\ref{main}(1), (2)    and~\ref{theorem::deg2}} \label{proofs1}

\begin{proof} [Proof  of Theorem ~\ref{theorem::deg2}]

 First we shall show that for any   odd natural numbers~$r,t$ with~$r\geq (t+1)(t+2)$,  a graph ~$G=G(t,r)$  given in Construction \ref{constr}  has no ~$\{t,r-t\}$-factor. Recall that $G$ is built of copies of $G'$ that is in turn constructed from copies of $H_i$.

 Assume there is a~$\{t,r-t\}$-factor~$F$ of~$G$. Consider one of the copies of $G'$.
 We shall show that in each of~$H_1, \ldots, H_{t+1}$ there is at least one edge incident to~$w$ that belongs to~$F$ and there 
 is at least one edge incident to~$w$ that does not belong to~$F$.

 Assume first  that in~$H_1$, none of the edges incident to~$w$ belongs to~$F$.
 Let~$x$ be the number of edges of~$M_1$ in~$F$.  
 The degree of each vertex in~$U_1$ in~$F$ is either ~$t$ or~$r-t$,  that is~$t$ modulo~$r-2t$.
 Thus the  number of edges of~$F$ between~$U_1$ and~$V_1$  modulo~$r-2t$   is ~$(r-1)t$.
 On the other hand (counting the edges incident to~$V_1$) it is ~$rt - 2x$ modulo~$r-2t$.
So, ~$ rt  - t =  rt - 2x   \pmod{r-2t}$   and  ~$2x -t  = 0 \pmod{r- 2t}.$
We also have that~$2x - t \geq  -t >  -r +2t$, so using the fact that ~$2x-t$  is odd, we have that~$2x -t = k(r-2t)$ for a positive integer ~$k$. Thus ~$2x-t \geq  r-2t$, i.e.,~$2x \geq r-t$, a contradiction since~$2x \leq r-t-2$ by construction.
 
 Assume further, that in~$H_1$ all of the edges incident to~$w$ belong to~$F$.
 Then  the number of edges of~$F$  between~$U_1$ and~$V_1$ is ~$(r-1)t =  rt - (t+2) - 2x  \pmod{r-2t}$.
 Then ~$2x+2 = 0 \pmod{r-2t}$.
Since~$2x+2~$ is an even positive number and ~$r-2t$ is odd, we have that 
$2x+2 \geq 2(r-2t) = 2r - 4t$, a contradiction since~$2x + 2 \leq r-t$.
 
 This shows that in each of the~$H_i$'s, there is an edge of~$F$ and a non-edge of~$F$ incident to~$w$. 
 Thus, the degree of~$w$ in~$F$ is at least~$t+1$ and at most~$r-t-1$, a contradiction.

~\\
Now, observe that  for any odd~$t$ and even~$r$, ~$K_{r+1}$ does not have a~$\{t,r-t\}$-factor.
Indeed,  since~$t$ is odd and~$r$ is even, a~$\{t, r-t\}$ factor has odd degrees. However,~$K_{r+1}$ has an odd number of vertices, 
 thus~$F$ is a graph with odd degrees and an odd number of vertices, a contradiction.
 \end{proof}

\vskip 0.2cm
 
\begin{proof}[Proof of Theorem \ref{main}(1)]

Note first that if~$r\in \{2,3\}$ then~$\chi(\HH)= \chicf(\HH)$ because any proper coloring is also a conflict-free 
coloring. Thus, we know precisely for what~$2$- and ~$3$-uniform  hypergraphs~$\chicf=\Delta+1$ from 
Theorem~\ref{theorem::BrooksHypergraphs}. In particular,~$f(2, \Delta) = \Delta+1$ and~$f(3,1)=2=\Delta+1$, but~$f(3,\Delta)\leq\Delta$ if~$\Delta\geq 2$.
 Theorem \ref{theorem::deg2} and  Lemma~\ref{lem::duality}  imply that ~$f(r,2)=3$ for all~$r\not\in\{3,5\}$.   \end{proof}   

\vskip 0.2cm

\begin{proof}[Proof of Theorem \ref{main}(2)] 
 Consider  $\HH$,   an $r$-uniform hypergraph with maximum degree $\Delta$ and maximum edge-degree $D$. 

{\bf Case 1} ~~ $\Delta \geq 2^{\frac{r}{2-r}} (er)^{\frac{r+2}{r-2}}$.  Then 
$er  \leq ( \Delta 2^{\frac{r}{r-2} })^{\frac{r-2}{r+2}}$. From Lemma~\ref{lem::LLLUpperChicf} $\chicf(\HH) \leq \frac{1}{2} (e r)^{1+ \frac{2}{r}} 
 \cdot\Delta^{\frac{2}{r}}\leq     \frac{1}{2}  ( \Delta 2^{\frac {r}{r-2}} )^{\frac{r-2}{r+2}\frac{r+2}{r} } \cdot\Delta^{\frac{2}{r}} \leq    \Delta $ .

{\bf Case 2}  ~~ $\Delta < 2^{\frac{r}{2-r}} (er)^{\frac{r+2}{r-2}} $. Then  $D\leq r\Delta \leq  r\cdot 2^{\frac{r}{2-r}} (er)^{\frac{r+2}{r-2}}\leq  
c' r^{2+ \frac{4}{r-2}}   \leq  2^{r/2}$  for sufficiently large $r$ and a constant $c'$. 
By Theorem \ref{KK}, $\chicf(\HH)\leq 120 \ln D \leq  120 (\ln c' + \frac{4}{r-2} \ln r)
 \leq 120 \ln c'  +  \ln r \leq \Delta$ since $\Delta \geq c_0 \ln r $. Here, $D$ is sufficiently large since $D \geq \Delta$.
\end{proof}

\vskip 0.2cm
\vskip 0.2cm

%%%%%%%%%%%%%%%%%%%%%%%%%%%%%%%%%%%%%%%%%
%%%%%%%%%%%%%%%%%%%%%%%%%%%%%%%%%%%%%%%%%
%%%%%%%%%%%%%%%%%%%%%%%%%%%%%%%%%%%%%%%%%
%%%%%%%%%%%%%%%%%%%%%%%%%%%%%%%%%%%%%%%%%

\section{Proof of Theorems \ref{main}(3) and \ref{r=4}} \label{proofs3}

We say that a sequence of edges $E_1, \ldots, E_k$ is an \emph{$E$-$E'$-path} of length $k$  if  $E=E_1,  E'=E_k$, and  $E_i\cap E_{i+1} \neq \emptyset $, $i=1, \ldots, k-1$.
We say that an $E$-$E'$-path is an $E$-$E'$-{\it geodesic} if this is an $E$-$E'$-path of smallest length. This length is also called a distance between $E$ and $E'$.
We call $|E_{k-1}\cap E_k |$ the {\it the tail-size  of the geodesic}.  A set $S$ of vertices \emph{separates} $\HH=(V,\E)$ if 
the hypergraph $\HH'=(V-S,  \{ E-S:  E\in \E\}) $ is disconnected. We denote $\HH'$ as $\HH-S$.
A hypergraph is called \emph{intersecting} if $E\cap E'\neq\emptyset$ for any pair of edges $E$ and $E'$.

\begin{lemma}\label{sep}
If $\HH$ is an $r$-uniform connected hypergraph, then there is a set $S$ of $r-1$ vertices contained in an edge of $\HH$ such that 
$\HH-S$ is connected.
\end{lemma}

\begin{proof}

Consider all pairs of edges with the largest distance between them. Among those, choose a pair $E, E'$ 
that has a geodesic $P'$ whose tail has the smallest size. Let  $P'=(E_1, \ldots, E_k)$.
Let  $v\in E_{k-1}\cap E_k$ and $S=E_k-\{v\}$.   We shall show that $S$ does not separate $\HH$.

If $k=2$, then  $\HH$ is intersecting.  Assume that $S$ separates $\HH$, and let $E$ be an edge 
such that  $E-S$ is in a connected component different  from  $E_1-S$ in $\HH-S$.
Then $v\not\in E$.  However, since $\HH$ is intersecting, $E\cap E_1 \neq \emptyset$ and $E\cap E_1 \subseteq E_1\cap S$.
We have that  $E_1\cap E_2 = (E_1 \cap S)\cup \{v\}$. Thus, $|E_1\cap E|<|E_1\cap E_2|$, a contradiction. Thus $S$ does not separate $\HH$.

We can assume that $k\geq 3$.
Note first that $S\cap E_i = \emptyset $ for $i= 1, \ldots, k-2$, otherwise $P'$ is not an $E$-$E'$-geodesic.
Assume that $\HH-S$ has at least two connected components. One of these contains all edges of $P'$ with 
$S$ deleted from them, and the other contains some vertex $u$.  Let $E''$ be an edge incident to $u$. Then an $E$-$E''$-geodesic $P''$  must contain  some 
vertices of $S$  in its edges and it must have length at most $k$.  
So,  $P''= (E_1', E_2', \ldots, E_i',  E_{i+1}', \ldots, E'')$, where $E_i' \cap S \neq \emptyset$ and $E_{i+1}'\cap S \neq \emptyset$.
Then $i\geq k-1$ because otherwise $E_1', \ldots, E_i'$ together with $E'$ contains an $E$-$E'$ path that is shorter than $P'$.   
Thus $i=k-1$ and $E_k'= E''$.  Then $P'$ together with $E''$  forms an $E$-$E''$-geodesic of length $k+1$, unless $E'' $ contains a vertex from $E_{k-1} \cap E_k$.  Since any geodesic in $\HH$ is of length at most $k$, we have that the latter happens and 
$E_1, E_2, \ldots, E_{k-1}, E''$ is an  $E$-$E''$-geodesic   of length $k$. 
However, $v\not\in E''$, so $|E_{k-1} \cap E''|< |E_{k-1} \cap E_k |$, where $E_k=E'$, a contradiction to the choice of $E, E'$ with the smallest  tail-size. 
Thus $S$ does not separate $\HH$.
\end{proof}

\begin{proof}[Proof of Theorem \ref{main}(3)]
We first shall show that $f(4, 3) \leq 3$. Then Lemma \ref{lem::inductDeltaColorable} implies that $f(4, \Delta)\leq \Delta$ for any $\Delta \geq 3$.
 Let~$\HH$ be  a connected $4$-uniform hypergraph, $\Delta(\HH)=3$  and $S$ be a subset of $3$ vertices contained in an edge $E$ of $\HH$ such that 
$\HH-S$ is connected. Such $S$ exists by Lemma \ref{sep}.
Lets order the vertices $V(\HH)$ as $(v_1,v_2,\ldots,v_n)$ such that 
$S=\{v_1, v_2, v_3\}$,    $\{v_n\} = E - \{v_1, v_2, v_3\}$ and for any vertex  $v_i \in \{v_4, \ldots, v_{n-1}\}$ 
there is at least one vertex $v_j$, $j>i$ such that $v_i$ and $v_j$ belong to the same edge. This ordering could be done since 
$\HH-S$ is connected. Then for each $j$,  $j<n$, $v_j$ is the last vertex of at most two edges and $v_n$ is the last vertex of 
at most three edges including $E$.
To find a good  coloring of $\HH$, we use a procedure:   Color  vertex~$v_i$ with a color~$i$, $i=1, 2, 3$.
Assume that $v_1, \ldots, v_{j-1}$ have been colored.  For $j\leq n$,  consider at most two edges $E_1, E_2$, such that 
$v_j$ is the last vertex in them and such that $ E\not\in \{E_1, E_2\}$  (for $j=n$).  
Let $q_i$ be the color of $X_i= E_i-\{v_j\}$ if $X_i$ is monochromatic and let $q_i$ be a unique color of $X_i$ otherwise, 
$i=1,2$.  Then let the color of $v_j$ be from $\{1,2,3\} -\{q_1, q_2\}$. This gives a good coloring of $E_1$ and $E_2$.  
When $j=n$ we need also to worry about the edge $E$ being good. However, since 
$E$ contains the vertices $v_1, v_2, v_3$ of colors $1, 2$ and $3$,   $E$ is good regardless of the color of $v_n$.
\end{proof}

Note however, that for any even $r\geq 2$, there is a $3$-regular hypergraph $\HH$ with every edge of size at most $r$ and $\chicf(\HH)>3$.
Indeed, for $r=2$ this is an ordinary graph $K_4$. For $r\geq 4$, consider $r/2$ vertex disjoint copies of $K_4-e$,  a complete graph on $4$ vertices without an edge,  and form an edge $E$ of size $r$ from all vertices of degree $2$ in these copies of $K_4-e$.  If the resulting hypergraph were to be conflict-free $3$-colorable, then all ordinary edges would be non-monochromatic and in each copy of $K_4-e$, the two vertices of degree $2$ would get the same color. 
Thus, in $E$, each color would be repeated.

\begin{proof}[Proof of Theorem \ref{r=4}]
Let $\HH$ be a connected $4$-uniform hypergraph. If $\Delta(\HH)\geq 3$, then by Theorem \ref{main}(3),  $\chicf(\HH)\leq \Delta$.
If $\Delta=1$ then  $\chicf(\HH) \geq 2 = \Delta +1$.
Assume that $\Delta =2$. It is easy to see that if $\HH$ is not regular, then $\chicf(\HH)\leq \Delta =2$.
Thus, we can assume that $\HH$ is a connected $2$-regular $4$-uniform hypergraph.
From Lemma \ref{lem::duality},  we see that $\chicf(\HH)=2$ iff its dual $4$-regular graph has a $\{1,3\}$-factor. 
Theorem 6.7 from \cite{Kano_factorBook} claims that a connected $4$-regular graph has a $\{1,3\}$-factor iff the number of vertices is even.
Thus $\chicf(\HH)=2$ iff the number $m$  of hyperedges is even.   If $m$ is odd, then $\chicf(\HH)\geq 3$ and $\chicf(\HH)\leq \Delta+1=3$, 
so $\chicf(\HH)=3$.
\end{proof}

\section{Conclusions and Open Questions}

The largest conflict free chromatic number of an $r$-uniform hypergraph with maximum degree $\Delta$ is 
denoted $f(r, \Delta)$.   As for the usual chromatic number, $f(r, \Delta) \leq  \Delta+1$.

Compared to graph/hypergraph Brooks'-type  theorems we show that the upper 
bound $\Delta+1$  is attained by  infinite classes of hypergarphs, 
in particular for each $r \neq 3,5$ and $\Delta =2$.
We prove that $f(r,\Delta)$ is at most $\Delta$ when $r$ is sufficiently large and $\Delta > c \ln r$.

We find a relation between  conflict-free colorings and   ~$\{t,r-t\}$-factors, studied as early as ~$1891$ by Petersen ~\cite{Petersen_RegGraphs}, when 
the existence of a ~$t$-factor in every~$r$-regular graph was proved for $r$ and $t$ even.
For each even~$r\neq 4$ and each odd~$t$ there is an $r$-regular  graph without~$\{t,r-t\}$-factor ~\cite{LuWangYu_factorsRegGraphs, Lu_ParityFactors}.   
 A~$4$-regular graph on~$n$ vertices has a~$\{1,3\}$-factor if and only if~$n$ is even.
 
For odd~$r$ and arbitrary even~$t$ or odd~$t\geq\frac{r}{3}$ every~$r$-regular graphs contains a~$\{t,r-t\}$-factor~\cite{Kano_factorsRegGraphs}.
Akbari and Kano conjectured the existence of such a factor for the remaining range of odd~$t$, see Conjecture \ref{conj}.
We disprove this  for~$r\geq (t+1)(t+2)$. For~$t=1$ this disproves the conjecture for all odd~$r\geq 7$.
An interesting open problem remains to determine whether a $\{1,4\}$-factor exists in any $5$-regular graph. 
As noted by by Akbari et al.~\cite{Akbari_MagicLabelings}, such a factor exists  if and only if there is an assignment  of labels $1$ and $2$ to the edges such that the sum of the labels on the edges incident to each vertex is zero modulo $3$.
There were at least two incomplete proofs for the existence of such a factor posted in the internet by independent groups of researchers.

Besides the value of $f(5,2)$, many other values for $f(r, \Delta)$ are to be determined. 
In particular,  the question of whether $f(r, \Delta ) \leq \Delta$ for any $r, \Delta \geq 3$ remains open except for the 
cases when $r\leq 4$ that we handled.

 \section{Acknowledgements}  The authors thank Torsten Ueckerdt,  Andr\'e K\"undgen, and Alexander Kostochka  for useful discussions.
 In addition, the authors gratefully acknowledge many useful comments of the referee.

\end{document}